\documentclass[11pt]{amsart}
\usepackage{amsthm}
\usepackage{array}
\usepackage{amsmath}
\usepackage{enumerate}
\usepackage{tikz}
\usetikzlibrary{calc}
\usepackage{amssymb}
\usepackage{latexsym}
\usepackage{amsfonts}
\usepackage{color}
\usepackage{mathrsfs}
\usepackage{epsfig,helvet}

\newcommand{\Z}{{\mathbb Z}}

\theoremstyle{plain} \numberwithin{equation}{section}
\newtheorem{thm}{Theorem}[section]
\newtheorem{theorem}[thm]{Theorem}
\newtheorem{lemma}[thm]{Lemma}
\newtheorem{corollary}[thm]{Corollary}

\newtheorem{definition}[thm]{Definition}
\newtheorem{proposition}[thm]{Proposition}

\begin{document}
\setcounter{page}{1}

\title[ Padhan, Hasan and Pradhan]{On the dimension of non-abelian tensor square of Lie superalgebras}

\author[Padhan]{Rudra Narayan Padhan}
\address{Centre for Data Science, Institute of Technical Education and Research  \\
	Siksha `O' Anusandhan (A Deemed to be University)\\
	Bhubaneswar-751030 \\
	Odisha, India}
\email{rudra.padhan6@gmail.com, rudranarayanpadhan@soa.ac.in}

	\author{IBRAHEM YAKZAN HASAN}
	\address{Centre for Applied Mathematics and Computing, Institute of Technical Education and Research  \\
	Siksha `O' Anusandhan (A Deemed to be University)\\
	Bhubaneswar-751030 \\
	Odisha, India}
\email{ibrahemhasan898@gmail.com}

\author[Padhan]{Sushree Sangeeta Pradhan}
\address{Department of Mathematics  \\
	C. V. Raman Global University\\
	Bhubaneswar-752054 \\
	Odisha, India}
\email{sushreesp1992@gmail.com}	

\maketitle

\begin{abstract}
In this paper, we determine upper bound for the non-abelian tensor product of finite dimensional Lie superalgebra. More precisely, if $L$ is a non-abelian nilpotent Lie superalgebra of dimension $(k \mid l)$ and its derived subalgebra has dimension $(r \mid s)$, then $ \dim (L\otimes L) \leq (k+l-(r+s))(k+l-1)+2$. We discuss the conditions when the equality holds for $r=1, s=0$ explicitly.
\end{abstract}

\section{Introduction}

%The theory of tensor products of Lie algebras has been developed by Ellis in \cite{Ellis1987,Ellis1991}. In  \cite{Rocco} Rocco proved that for every finite $p$-group of order $p^n$ with derived subgroup of order $p^m$, the order of its tensor square is at most $p^{n(n-m)}$. Recently, P. Niroomand \cite{Niroomand}, improved and generalized this upper bound  to obtain a similar result for a non-abelian nilpotent Lie algebra of finite dimension. It is a natural problem to extend known results for tensor products of Lie algebra to tensor products of Lie superalgebras. Such results, have been obtained in \cite{GKL2015, Pilar,Padhandetec,Nayak2018}. \cite{Pilar} defined and extended the notation and proprieties of Whitehead’s quadratic functor to lie superalgebras. The purpose of this article is to show that the similar result in \cite{Niroomand} holds for non-abelian nilpotent Lie superalgebras of finite dimension.

In 1991, Rocco \cite{Rocco} proved that if $G$ is a finite $p$-group of order $P^n$ with derived subgroup of order $p^m$, then $G\otimes G \leq p^{n(n-m)}$. Further this bound was imporved by Niroomand \cite{p2010}, i.e., $G\otimes G \leq p^{n(n-m)}$. Ellis \cite{Ellis1987,Ellis1991} developed the theory of tensor product for Lie algebra. If $L$ and $K$ are two finite dimensional nilpotent Lie algebras, then the upper bound and lower bound of the dimension of the non-abelian tensor product $L\otimes K$ have been stuided by salemkar et. al \cite{s2010} and the results are generalization of Rocco's results for finite $p$-group.  Recently, an improved upper bound on the dimension of  $L\otimes L$ has been discussed in \cite{Niroomand}, explicitly, if $L$ is a non-abelian nilpotent Lie algebra of dimension $n$ and its derived subalgebra has dimension $m$, then $ \dim (L\otimes L) \leq (n-m)(n-1)+2$.

\smallskip 

Lie superalgebras have applications in many areas of Mathematics and Theoretical Physics as they can be used to describe supersymmetry. Kac \cite{Kac1977} gives a comprehensive description of mathematical theory of Lie superalgebras, and establishes the classification of all finite dimensional simple Lie superalgebras over an algebraically closed field of characteristic zero. In the last few years, the theory of Lie superalgebras has evolved remarkably, obtaining many results in representation theory and classification. Most of the results are extension of well known facts of Lie algebras \cite{MC2011, Musson2012, GKL2015}.  Recently, Garc\'{i}a-Mart\'{i}nez \cite{GKL2015} introduced the notions of non-abelian tensor product of Lie superalgebras and exterior product of Lie superalgebras over a commutative ring. In this paper we determine an upper bound on the dimension of non-abelian tensor product of nilpotent Lie superalgebra.

\section{Preliminaries and Auxiliary Results}

Let $\mathbb{Z}_{2}=\{\bar{0}, \bar{1}\}$ be a field. A $\mathbb{Z}_{2}$-graded vector space $V$ is simply a direct sum of vector spaces $V_{\bar{0}}$ and $V_{\bar{1}}$, i.e., $V = V_{\bar{0}} \oplus V_{\bar{1}}$. It is also referred as a superspace. We consider all vector superspaces and superalgebras are over field $\mathbb{F}$ (characteristic of $\mathbb{F} \neq 2,3$). Elements in $V_{\bar{0}}$ (resp. $V_{\bar{1}}$) are called even (resp. odd) elements. Non-zero elements of $V_{\bar{0}} \cup V_{\bar{1}}$ are called homogeneous elements. For a homogeneous element $v \in V_{\sigma}$, with $\sigma \in \mathbb{Z}_{2}$ we set $|v| = \sigma$ as the degree of $v$. A  subsuperspace (or, subspace)  $U$ of $V$ is a $\mathbb{Z}_2$-graded vector subspace where  $U= (V_{\bar{0}} \cap U) \oplus (V_{\bar{1}} \cap U)$. We adopt the convention that whenever the degree function appears in a formula, the corresponding elements are supposed to be homogeneous. 

\smallskip 

A  Lie superalgebra (see \cite{Kac1977, Musson2012}) is a superspace $L = L_{\bar{0}} \oplus L_{\bar{1}}$ with a bilinear mapping $ [., .] : L \times L \rightarrow L$ satisfying the following identities:
\begin{enumerate}
\item $[L_{\alpha}, L_{\beta}] \subset L_{\alpha+\beta}$, for $\alpha, \beta \in \mathbb{Z}_{2}$ ($\mathbb{Z}_{2}$-grading),
\item $[x, y] = -(-1)^{|x||y|} [y, x]$ (graded skew-symmetry),
\item $(-1)^{|x||z|} [x,[y, z]] + (-1)^{ |y| |x|} [y, [z, x]] + (-1)^{|z| |y|}[z,[ x, y]] = 0$ (graded Jacobi identity),
\end{enumerate}
for all $x, y, z \in L$. Clearly $L_{\bar{0}}$ is a Lie algebra, and $L_{\bar{1}}$ is a $L_{\bar{0}}$-module. If $L_{\bar{1}} = 0$, then $L$ is just Lie algebra, but in general a Lie superalgebra is not a Lie algebra.  A Lie superalgebra $L$, is called abelian if  $[x, y] = 0$ for all $x, y \in L$. Lie superalgebras without even part, i.e., $L_{\bar{0}} = 0$, are  abelian. A subsuperalgebra (or subalgebra) of $L$ is a $\mathbb{Z}_{2}$-graded vector subspace which is closed under bracket operation. The graded subalgebra $[L, L]$,  of $L$  is known as the derived subalgebra of $L$. A $\mathbb{Z}_{2}$-graded subspace $I$ is a graded ideal of $L$ if $[I, L]\subseteq I$. The ideal 
\[Z(L) = \{z\in L : [z, x] = 0\;\mbox{for all}\;x\in L\}\] 
is a graded ideal and it is called the {\it center} of $L$. A homomorphism between superspaces $f: V \rightarrow W $ of degree $|f|\in \mathbb{Z}_{2}$, is a linear map satisfying $f(V_{\alpha})\subseteq W_{\alpha+|f|}$ for $\alpha \in \mathbb{Z}_{2}$. In particular, if $|f| = \bar{0}$, then the homomorphism $f$ is called homogeneous linear map of even degree. A Lie superalgebra homomorphism $f: L \rightarrow M$ is a  homogeneous linear map of even degree such that $f([x,y]) = [f(x), f(y)]$ holds for all $x, y \in L$.  If $I$ is an ideal of $L$, the quotient Lie superalgebra $L/I$ inherits a canonical Lie superalgebra structure such that the natural projection map becomes a homomorphism. The notions of {\it epimorphisms, isomorphisms} and {\it automorphisms} have the obvious meaning. 
\smallskip

Throughout this article, for superdimension of Lie superalgebra $L$ we simply write $\dim L=(m\mid n)$, where $\dim L_{\bar{0}} = m$ and $\dim L_{\bar{1}} = n$. Also  $A(m \mid n)$ denotes an abelian Lie superalgebra where $\dim A=(m\mid n)$. A  Lie superalgebra $L$ is said to be Heisenberg Lie superalgebra if $Z(L)=L'$ and $\dim Z(L)=1$. According to the homogeneous generator of $Z(L)$, Heisenberg Lie superalgebras can further split into even or odd Heisenberg Lie superalgebras \cite{MC2011}. By Heisenberg Lie superalgebra we mean special Heisenberg Lie superalgebra in this article, for more details on Heisenberg Lie superalgebras and its multiplier see \cite{Nayak2018,SN2018b,Padhandetec,p50, p54,p55,hesam}. Now we list some useful results from \cite{Nayak2018}, for further use.

%\begin{theorem}\label{th3.3}\cite[See Theorem 3.4]{Nayak2018}
%	\[\dim \mathcal{M}(A(m \mid n)) = \big(\frac{1}{2}(m^2+n^2+n-m)\mid mn \big).\]
%\end{theorem}

\begin{theorem} \label{th3.4}\cite[See Theorem 4.2, 4.3]{Nayak2018}
	Every Heisenberg Lie superalgebra with even center has dimension $(2m+1 \mid n)$ and  is isomorphic to $H(m , n)=H_{\overline{0}}\oplus H_{\overline{1}}$, where
	\[H_{\overline{0}}=<x_{1},\ldots,x_{2m},z \mid [x_{i},x_{m+i}]=z,\ i=1,\ldots,m>\]
	and 
	\[H_{\overline{1}}=<y_{1},\ldots,y_{n}\mid [y_{j}, y_{j}]=z,\  j=1,\ldots,n>.\]
	Further,
	$$
	\dim \mathcal{M}(H(m , n))=
	\begin{cases}
		(2m^{2}-m+n(n+1)/2-1 \mid 2mn)\quad \mbox{if}\;m+n\geq 2\\
		(0 \mid 0) \quad \mbox{if}\;m=0, n=1\\  
		(2 \mid 0)\quad \mbox{if}\;m=1, n=0.
	\end{cases}
	$$
\end{theorem}

The following is the established result for the multiplier and cover of Heisenberg Lie superalgebra of odd center.

\begin{theorem}\label{th3.6}\cite[See Theorem 2.8]{SN2018b}
	Every Heisenberg Lie superalgebra, with odd center has dimension $(m \mid m+1)$, is isomorphic to $H_{m}=H_{\overline{0}}\oplus H_{\overline{1}}$, where
	\[H_{m}=<x_{1},\ldots,x_{m} , y_{1},\ldots,y_{m},z \mid [x_{j},y_{j}]=z,   j=1,\ldots,m>.\]
	Further,
	$$
	\dim \mathcal{M}(H_{m})=
	\begin{cases}
		(m^{2}\mid m^{2}-1)\quad \mbox{if}\;m\geq 2\\  
		(1\mid 1) \quad \quad \mbox{if}\;m=1.  
	\end{cases}
	$$
\end{theorem}

 \section{ Non-abelian tensor products}
Here we recall some of the known notation and results from \cite{GKL2015}. Let $P$ and $M$ be two Lie superalgebras, then by an action of $P$ on $M$ we mean a $\mathbb{K}$-bilinear map of even degree $P\times M \longrightarrow M,$ \[(p, m)\mapsto {}^ pm, \]
such that 
\begin{enumerate}
\item ${}^{[p , p']} m = {}^p({}^{p'}m)-(-1)^{|p||p'|}~  {}^{p'}({}^{p}m),$
\item ${}^p{[m , m']}=[{}^pm, m']+(-1)^{|p||m|}[m , {}^pm'],$
\end{enumerate}
for all $p, ~p' \in P$ and $m, ~m' \in M$. For any Lie superalgebra $M$, the Lie multiplication induces an action on itself via ${}^mm'=[m , m']$.  The action of $P$ on $M$ is called trivial if ${}^pm=0$ for all $p \in P$ and $m \in M$. \\
 Given two Lie superalgebras $M$ and $P$ with action of $P$ on $M$, we define the semidirect product $M \rtimes P$ with underlying supermodule $M \oplus P$ endowed with the bracket given by
 $[(m, p), (m', p')]=([m, m']+ {}^pm'-(-1)^{|m||p'|}({}^{p'}m), [p, p'])$. A crossed module of Lie superalgebras is a homomorphism of Lie superalgebras $\partial : M\longrightarrow P$ with an action of $P$ on $M$ satisfying 
 \begin{enumerate}
 	\item $\partial ({}^p{m})=[p,\partial(m)],$
 	
 	\item ${}^{\partial(m)}{m'}=[m, m'],$ for all $p \in P$ and $m, m' \in M$.
 	
 \end{enumerate} 

 	A bilinear function $f:M\times N\longrightarrow T$ is called Lie superpairing if the following relations are satidfied:
 	\begin{enumerate}
 		\item $f([m , m'], n)= f(m, {}^ {m'}{n}) -(-1)^{|m||m'|}f(m', {}^{m}{n})$,
 		\item  $f(m, [n,n'])=(-1)^{|n'|(|m|+|n|)}f({}^{n'}{m}, n)-(-1)^{|m||n|}f({}^{n}{m}, n')$,
 		\item $f({}^{n}{m} , {}^{m'}{n'})=-(-1)^{|m||n|}[f(m, n),f(m', n')]),$  
 	\end{enumerate} 
 	for every $m, m' \in M_{\overline{0}}\cup  M_{\overline{1}}$ and $n , n' \in N_{\overline{0}}\cup  N_{\overline{1}}$.

 Let $M$ and $N$ be two Lie superalgebras with actions on each other. Let $X_{M , N}$ be the $\mathbb{Z}_{2}$-graded set of all symbols $m\otimes n$, where $m \in M_{\overline{0}}\cup  M_{\overline{1}}$, $n \in N_{\overline{0}}\cup  N_{\overline{1}}$ and the $\mathbb{Z}_{2}$-gradation is given by $|m\otimes n|=|m|+|n|$. The non-abelian tensor product of $M$ and $N$, denoted by $M \otimes N$, as the Lie superalgebra generated by $X_{M , N}$  and subject to the relations:
 \begin{enumerate} 
 	\item $\lambda (m \otimes n)=\lambda m \otimes n= m \otimes \lambda n$,
 	\item $(m + m')\otimes n= m\otimes n +m'\otimes n$,  where $m , m'$ have the same degree,\\
 	$m \otimes (n + n')=m \otimes n + m \otimes n'$,  where $n , n'$ have the same degree,
 	
 	\item $[m , m']\otimes n= (m\otimes {}^ {m'}{n}) -(-1)^{|m||m'|}(m'\otimes {}^{m}{n})$,\\
 	$m\otimes [n,n']=(-1)^{|n'|(|m|+|n|)}({}^{n'}{m}\otimes n)-(-1)^{|m||n|}({}^{n}{m}\otimes n')$,
 	
 	\item $[m\otimes n,m'\otimes n']=-(-1)^{|m||n|}({}^{n}{m} \otimes {}^{m'}{n'}),$      
 \end{enumerate}
 for every $\lambda \in \mathbb{K}, m, m' \in M_{\overline{0}}\cup  M_{\overline{1}}$ and $n , n' \in N_{\overline{0}}\cup  N_{\overline{1}}$. The tensor product $M \otimes N$ has $\mathbb{Z}_{2}$-grading given by $(M \otimes N)_{\alpha}=\oplus_{\beta+\gamma=\alpha}  (M_{\beta}+N_{\gamma})$ for $\alpha, \beta, \gamma \in \Z_2$. If $M=M_{\overline{0}}$ and $N=N_{\overline{0}}$ then $M \otimes N$ is the non-abelian tensor product of Lie algebras introduced and studied in \cite{Ellis1991}.
 
 Actions of Lie superalgebras $M$ and $N$ on each other are said to be compatible if 
 \begin{enumerate}
 	\item ${}^{({}^{n}{m})}{n'}=-(-1)^{|m||n|}[{}^{m}{n},n']$,
 	\item ${}^{({}^{m}{n})}{m'}=-(-1)^{|m||n|}[{}^{n}{m},m']$,
 \end{enumerate}
 for all $m, m' \in M_{\overline{0}}\cup  M_{\overline{1}}$ and $n, n' \in N_{\overline{0}}\cup  N_{\overline{1}}$.
 For instance if $M$, $N$ are two graded ideals of a Lie superalgebra then the actions induced by the bracket are compatible.
 
 \smallskip

 %\begin{proposition}\label{prop4}\cite[Proposition 3.8]{GKL2015}
 %	Given a short exact sequence of Lie superalgebras
% 	$$(0, 0) \longrightarrow (K,L) \overset{(i, j)} \longrightarrow (M, N) \overset{(\phi, \psi)} \longrightarrow (P, Q) \longrightarrow (0, 0)$$ there is an exact sequence of Lie superalgebras
% 	$$
% 	(K \otimes M) \rtimes (M \otimes L) \overset{\alpha} \longrightarrow M \otimes N \overset{\phi \otimes \psi} \longrightarrow P \otimes Q \longrightarrow 0.$$
% \end{proposition}
% Specifically given a Lie superalgebra $M$ and a graded ideal $K$ of $M$ there is an exact sequence 
% \begin{equation}\label{eq1}
% 	(K \otimes M) \rtimes (M \otimes K) \longrightarrow M \otimes M \longrightarrow (M/K) \otimes (M/K) \longrightarrow 0.
 %\end{equation}

We will denote the $\mathbb{K}$-module tensor product of $M$ and $N$ as $M \otimes_{mod} N$. The following result comes immediately from \cite[Proposition 3.8]{GKL2015}. 

 \begin{proposition}\label{propp3.2}
 	If $N\subseteq	L^2\cap Z(G)$, then the sequence
 	$N\otimes L\longrightarrow L\otimes L \longrightarrow L/N \otimes L/N\longrightarrow 0$
is exact.
 \end{proposition}

 \begin{proposition}\label{pro3333}\cite[Proposition 3.5]{GKL2015}
If $M$ and $N$ act trivially on each other, then there is an isomorphism  \[M\otimes N \cong M^{ab}\otimes _{mod}N^{ab}.\]
where $M^{ab}=M/[M,M]$ and $N^{ab}=N/[N,N].$
 \end{proposition}

\begin{proposition}\label{prop4a}\cite[Proposition 3.6]{Padhandetec}
 Suppose the Lie superalgebras $L$ and $M$ are acting trivially on each other, then
 	\[(L \oplus M) \otimes (L \oplus M) \cong (L \otimes L) \oplus (L\otimes M) \oplus (M \otimes L) \oplus (M \otimes M).\]
 \end{proposition}
 
Consider a Lie superalgebra $M = M_{\bar{0}}\oplus M_{\bar{1}}$ and the identity map on $M$, i.e., the map $id : M \longrightarrow M$. The map $id$ is a crossed module, and hence the exterior square  $M\wedge M$  is obtained from $M\otimes M$ by imposing the additional relations:
 \begin{enumerate}
 	\item $m \otimes m'+   (-1)^{|m||m'|}(m'\otimes m)=0$
 	\item $m_{\overline{0}} \otimes m_{\overline{0}}=0$
 \end{enumerate}
 with $m,m' \in M_{\overline{0}} \cup M_{\overline{1}}$, and $m'\wedge m $ is the image of $m'\otimes m$.

 \begin{lemma}\label{lem4b}\cite[Lemma 6.1]{GKL2015}
 Let $M \square M$ be the submodule of $M \otimes M$ generated by elements
 \begin{enumerate}
 		\item $m \otimes m'+ (-1)^{|m||m'|}(m'\otimes m)$, 
 		\item $m_{\overline{0}} \otimes m_{\overline{0}}$,
 	\end{enumerate} 
with $m,m' \in M_{\overline{0}} \cup M_{\overline{1}},~m_{0} \in M_{\overline{0}}$. Then $M \square M$ is a central graded ideal of $M \otimes M$.
 \end{lemma}

 \begin{definition}
The exterior product of $M$ and $M$ is denoted as $M \wedge M$ and is defined as the quotient Lie superalgebra \[M\wedge M= \frac{M\otimes M}{ M\square M}.\] 
 \end{definition}
For any $m \otimes m' \in M \otimes M$ we denote the coset $m \otimes m' +M \square M$ by $m \wedge m'$.

\smallskip

The universal quadratic functor of Lie superalgebra  was  introduced in \cite{Pilar}. The quadratic functor  helps to establish the relations between the Lie exterior product and the Lie tensor product of Lie superalgebras. Here we recall the definition of universal quadratic functor of Lie superalgebra and some results from \cite{Pilar} which will be useful in the next section.

\begin{definition}
Let $L$ be a supermodule. We define the supermodule $\Gamma(M)$ as the direct sum
	\[\Gamma(M)=R^{M_{\bar{0}}} \oplus (M\otimes M) 
	=\{\gamma(m_{\bar{0})}+ m'\otimes m'' \lvert m_{\bar{0}} \in M_{\bar{0}}, \, m',m'' \in M\},\]
and subject to the homogeneous relations
	\begin{equation}
		\gamma(\lambda m_{\bar{0}})=\lambda^2\gamma(m_{\bar{0}})
	\end{equation}
	
	\begin{equation}
		\gamma(m_{\bar{0}}+ m'_{\bar{0}})-\gamma(m_{\bar{0}})-\gamma(m'_{\bar{0}})= m_{\bar{0}}\otimes m'_{\bar{0}}
	\end{equation}
	\begin{equation}
		m\otimes m'= (-1)^{|m||m'|} m'\otimes m
	\end{equation}

	\begin{equation}
		m_{\bar{1}}\otimes m_{\bar{1}}=0
	\end{equation}
where $\lambda \in \mathbb{K}$, $m_{\bar{0}}, m'_{\bar{0}} \in M_{\bar{0}}$, $	m_{\bar{1}} \in 	M_{\bar{1}}$ and $m,m' \in M$ with the induced grading.
\end{definition}

%Let $(M, P , \partial)$ and $(N, P , \sigma)$ be two crossed modules of Lie superalgebras, and
%consider the following graded Lie subalgebra of $M\oplus N$:
%\[M\otimes_P N=\{m+n \in M\oplus N \lvert \partial(m)=\sigma(n) \}\]
%and let 
%\[< M,N >= \{-(-1)^{|m||n|} %\,\,\,{}^n{m}+{}^m{n}\}. \]

%It is easy to check that $< M,N >$ is a graded ideal of $M\otimes_{mod} N$, and also that the quotient
%$Q:=\frac{M\otimes_{mod} N}{< M,N >} $ 
%is abelian. We will denote its elements by $\overline{m+n}$.

%\begin{proposition}\cite[See proposition 7.2.8]{Pilar}
%	Let $(M, P , \partial)$ and $(N, P , \sigma)$ be two crossed modules of Lie superalgebras,
%	and define $\psi: \Gamma(Q) \longrightarrow M\otimes N$ by
%	\[  \psi(\gamma(\overline{\mu_{\bar{0}}+\nu_{\bar{0}}}) + \overline{m+n}\otimes\overline{m'+n'})=m_{\bar{0}}\otimes n_{\bar{0}}+ m\otimes n'+ (-1)^{|m||n|} m'\otimes n\]
	
%	Then, the following sequence of Lie superalgebras
%	$$ \Gamma(Q) \overset{\psi}\longrightarrow M\otimes N \overset{\pi}\longrightarrow M \wedge N \longrightarrow 0$$
%	is exact, where $\pi(m\otimes n)=m\wedge n.$ 
	
%\end{proposition}

\begin{proposition}
For any Lie superalgebra $M$, there exists an exact sequence
	\begin{equation}\label{eq31}
		\Gamma(M^{ab}) \overset{\psi}\longrightarrow M\otimes M \overset{\pi}\longrightarrow M \wedge M \longrightarrow 0.
	\end{equation}
Also, given two graded ideals $I$ and $J$ of $M$, the following sequence is exact:
\begin{equation}
	\Gamma(\frac{I\cap J}{[I,J]}) \overset{\psi}\longrightarrow I\otimes J \overset{\pi}\longrightarrow I \wedge J \longrightarrow 0.
\end{equation}
\end{proposition}

\begin{proposition}\cite[ Proposition 7.2.4]{Pilar}\label{pro310}
Let $M$ and $N$  be two supermodules. Then 
	\[ \Gamma(M\oplus N)\cong \Gamma(M)\oplus\Gamma(N) \oplus (M\otimes_{mod} N). \]
\end{proposition}

\begin{definition}
A quadratic map between two supermodules $M$ and $N$ is a map 
	$\varphi: M \longrightarrow N$ satisfying:
	\begin{enumerate}
		\item $\varphi(\lambda m)=\lambda^2\varphi(m)$.
		\item The associated symmetric function $b_\varphi: M\times M\longrightarrow N$ defined by $b_\varphi(m,m')=\varphi(m+m')-\varphi(m)-\varphi(m')$ is bilinear.
\end{enumerate}
 \end{definition}

\begin{proposition}\cite[ Proposition 7.2.2]{Pilar}
The pair $\gamma = \gamma_M=(\gamma_{\bar{0}},b_\gamma)$, with $\gamma_{\bar{0}}:M_{\bar{0}}\longrightarrow \Gamma(M_{\bar{0}})$  defined by $\gamma_{\bar{0}}(x_{\bar{0}})= \gamma(x_{\bar{0}})$
	and $b_\gamma:M\times M \longrightarrow \Gamma(M)$  defined by $b_\gamma(x,y)=x\otimes y$  is a quadratic	map.
\end{proposition}

\begin{proposition}\label{th3.3}\cite[Proposition 7.2.6]{Pilar}
Let $M$ be a free supermodule, and let $\{\bar{x_i}\}_{i\in I_{\bar{0}}}\bigcup \{\bar{x_i}\}_{i\in I_{\bar{1}}} $  be an ordered basis of $M$ composed by homogeneous elements and such that the elements $\{\bar{x_i}\}_{i\in I_{\bar{0}}}$ of $M_{\bar{0}}$ are less or equal than those of $M_{\bar{1}}$, $\{\bar{x_i}\}_{i\in I_{\bar{1}}}$. Then $\Gamma(M)$ is free with basis 
	$$\{\gamma_{\bar{0}}(x_i)\}_{i\in I_{\bar{0}}}\bigcup \{b_{\gamma}(x_i,x_j)\}_{i,j\in I_{\bar{0}}\bigcup I_{\bar{1}}, i< j} .$$
\end{proposition}

\begin{proposition}\label{th3.12}\cite[Proposition 7.2.9]{Pilar}
Let $M$ be a Lie superalgebra such that $M^{ab}$ is a free supermodule. Then the homomorphism $\psi$ in the sequence \ref{eq31} is injective.
\end{proposition}

\begin{corollary}\label{cor313}
Let $M$ be a Lie superalgebra such that $M^{ab}$ is a free supermodule. Then the following sequence is exact 
	$$ 0\longrightarrow \Gamma(M^{ab}) \overset{\psi}\longrightarrow M\otimes M \overset{\pi}\longrightarrow M \wedge M \longrightarrow 0.$$
	
\end{corollary}
% Let $t(L)=\frac{1}{2}[(m+n)^2+(n-m)] - \dim \mathcal{M}(L)\geq 0$. 
 % We list here all finite dimensional nilpotent Lie superalgebras $ L$ with corank $t(L) \leq 3$. 

%\begin{theorem}\cite [See Theorem 5.1]{SN2018b}\label{th10}
%	Let $L$ be a non-abelian nilpotent Lie superalgebra, $\dim L=(m \mid n)$. Then
	
%	\begin{enumerate}
%		\item $t(L)=0$ if and only if $L \cong A(m, n)$.
%		\item $t(L)=1$ if and only if $L \cong H(1, 0)$.
%		\item $t(L)=2$ if and only if $L \cong H(1, 0)\oplus A(1 \mid 0)$ or $H(0,1)$.
%		\item $t(L)=3$ if and only if $L \cong H(1, 0)\oplus A(2 \mid 0)$, $H(0,1)\oplus A(1 \mid 0)$, $H(0,2)$, $H_{1}$ or $H(0,1)\oplus A(0 \mid 1)$.
%		\item $t(L)=4$ if and only if $L \cong H(1, 0)\oplus A(3 \mid 0)$, $L_{5,0}^2$, $L_{4,0}^1$, $H(0,3)$, $H_{1} \oplus A(1 \mid 0)$, $H(0,1)\oplus A(1 \mid 1)$, $H(0,1)\oplus A(0 \mid 2)$, $H(0, 1)\oplus A(2 \mid 0),\, H_{1}\oplus A(0 \mid 1)$  or $H_{1} \oplus A(2 \mid 0)$.
%	\end{enumerate}
%\end{theorem}

\section{Main results}
In this section, we show that any $(k \mid l)$-dimensional non-abelian nilpotent Lie superalgebra $L$ with derived supersubalgebra of dimension $(r \mid s)$ satisfies
 $\dim (L\otimes L) \leq (k+l-(r+s))(k+l-1)+2.$
In particular, for $r=1, s=0$, we determine the structure of $L$ when the equality holds.

\begin{lemma}\label{lemm4.2}
	For any two abelian Lie superalgebras $M$ and $N$ there is an isomorphism
	\[(M \oplus N) \square (M \oplus N) \cong (M \square M) \oplus (N\square N) \oplus (M \otimes N) .\]
\end{lemma}
\begin{proof}
	The proof follows from Proposition \ref{pro310} and Lemma \ref{lem41}.
\end{proof}

\begin{lemma}\label{lem41}
Let $M$ be be a finite dimension Lie superalgebra. Then \[\Gamma(M/M^2)\cong M/M^2\square M/M^2.\]
\end{lemma}
\begin{proof}
From the exact Sequence \ref{eq31}  we have Im$\psi =$Ker $\pi = M\square M$. Now we shall prove that 
	\begin{equation}\label{eq4.1}
		\Gamma(A(m\mid n))\cong A(m\mid n)\square A(m\mid n)
	\end{equation}
Let $m+n=1$, then by Corollary \ref{cor313} we have 
	\[\Gamma(A(1\mid 0))\cong A(1\mid 0)\square A(1\mid 0)\] 	
	\[\Gamma(A(0\mid 1))\cong A(0\mid 1)\square A(0\mid 1)\]
Now assume that $m+n\geq 2$. Then from Proposition \ref{pro3333}, Lemma \ref{lem41} and using induction hypothesis we have 
	\begin{align*}
		\Gamma (A(m+1\mid n))&\cong \Gamma(A(m\mid n)\oplus A(1\mid 0)) 
		\\
		&\cong\Gamma(A(m\mid n))\oplus\Gamma(A(1\mid 0))\oplus A(m\mid n)\otimes A(1\mid 0)\\
		&\cong  A(m\mid n)\square A(m\mid n) \oplus  A(1\mid 0) \square A(1\mid 0)  \oplus A(m\mid n)\otimes A(1\mid 0)\\
		&\cong  (A(m\mid n)\oplus A(1\mid 0))\square (A(m\mid n)\oplus A(1\mid 0))\\
		&\cong  A(m+1\mid n)\square A(m+1\mid n).
	\end{align*}
Similarly, we can see that $\Gamma (A(m\mid n+1))\cong  A(m\mid n+1)\square A(m\mid n+1).$ Thus the Relation \ref{eq4.1} is true for all $m+n\geq1.$
\end{proof}

\begin{corollary}\label{cor4.3}
Let $M$ be a finite dimension Lie superalgebra, then
	\[M \square M\cong M/M^2\square M/M^2\cong \Gamma(M/M^2).\]	
\end{corollary}
\begin{proof}
From Lemma \ref{lem41} and the exact Sequence \ref{eq31}, we have the following epimorphism   $M/M^2\square M/M^2\longrightarrow M\square M$. Now by using  Lemma \ref{lem41} and the natural epimorphism  $M\square M\longrightarrow M/M^2\square M/M^2$ we get the desire result.
\end{proof}

\begin{proposition}\label{prop444}
For  $m+n\geq 2$, $H(m,n)\otimes H(m,n)\cong H(m,n)/H^2(m,n) \otimes H(m,n)/H^2(m,n).$
Moreover, $H(1, 0) \otimes H(1, 0) \cong A (6 \mid 0)$ and $H(0, 1) \otimes H(0, 1) \cong A(1 \mid 0)$.
\end{proposition}

\begin{proof}
Suppose that $m+n\geq 2$. From Theorem \ref{th3.3} and Lemma \ref{lem41}, we have \[\dim (H(m,n)\square H(m,n))=\dim \Gamma(H(m,n)/ H^2(m,n))=(2m^2+m +\frac{n(n-1)}{2}\mid 2mn),\] 
and by Theorem \ref{th3.4} we have $\dim \mathcal{M}(H(m,n))=(2m^{2}-m+n(n+1)/2-1 \mid 2mn)$. Then 
\begin{align*}
	\dim (H(m,n)\otimes H(m,n)) & 
	=\dim (H(m,n)\square H(m,n))+ \dim \mathcal{M}(H(m,n))+\dim H^2(m,n)\\
	&=(2m^2+m+\frac{n(n-1)}{2} \mid 2mn)+(2m^{2}-m+n(n+1)/2-1 \mid 2mn)+(1 \mid 0)\\
	&=(4m^2+n^2 \mid 4mn)\\
	&=4m^2+4mn+n^2.
\end{align*}
On the other hand, $\dim H(m,n)/H^2(m,n) \otimes H(m,n)/H^2(m,n)=(2m+n)^2= 4m^2+4mn+n^2$. Now the result can be obtained by the natural epimorphism $H(m,n)\otimes H(m,n)\longrightarrow H(m,n)H^2(m,n) \otimes H(m,n)H^2(m,n)$.

\smallskip

Similarly one can see that $H(1, 0) \otimes H(1, 0)$, (and $H(0, 1) \otimes H(0, 1)$ ) is Lie superalgebra with dimension $(6 \mid 0)$, (resp. of dimension $(1 \mid 0)$). One can easily check that $H(1, 0) \otimes H(1, 0)$ and $H(0, 1) \otimes H(0, 1)$ are abelian too.
\end{proof}
Now we state the same result for a Heisenberg Lie superalgebra with odd center, as the proof is quite similar to the Proposition \ref{prop444}, so we omit it.

\begin{proposition}\label{prop45}
When $m\geq 2, ~H_m\otimes H_m\cong H_m/H^2_m \otimes H_m/H^2_m$. Moreover, $H_1 \otimes H_1 \cong A (2 \mid 3)$.
\end{proposition}

\begin{theorem}
Let $L$ be an $(k \mid l)$-dimensional non-abelian nilpotent Lie superalgebra
with derived supersubalgebra of dimension $(r \mid s)$. Then
	   \[ \dim (L\otimes L) \leq (k+l-(r+s))(k+l-1)+2. \]
In particular, for $r=1, s=0$ the equality holds if and only if $L\cong H(1,0)\oplus A(k-3 \mid l)$.
\end{theorem}

\begin{proof}
First, for $r+s=1$, we have the following cases 
     \begin{enumerate}
     	\item $r=1,~ s=0;$ 
     	\item $r=0,~ s=1.$
     \end{enumerate} 
(1) Let $r=1,~ s=0$, then from \cite[Proposition 3.4]{SN2018b}, we have $L \cong H(m,n)\oplus A(k-2m-1 \mid l-n)$ for $m+n\geq 1$. Thus by Proposition \ref{prop4a}, 
$$\dim (H(m,n)\oplus A(k-2m-1 \mid l-n)) \otimes (H(m,n)\oplus A(k-2m-1 \mid l-n)) = \dim (H(m,n) \otimes H(m,n)) $$   $$+ 2\dim (H(m,n)\otimes_{mod} A(k-2m-1 \mid l-n))  + \dim(A(k-2m-1 \mid l-n) \otimes A(k-2m-1 \mid l-n)).$$
From Proposition \ref{pro3333},
 \[(H(m,n)\otimes A(k-2m-1 \mid l-n))\cong (H(m,n)/H^2(m,n)\otimes_{mod} A(k-2m-1 \mid l-n)).\]
Now Consider the following cases:\\

(i) When $m=1,n=0$, then by using Proposition \ref{prop444},  
\begin{align*}
  \dim (H(1,0)\oplus A(k-3 \mid l)) \otimes (H(1,0)\oplus A(k-3 \mid l))
          & =(6 \mid 0) + 2(2\mid 0)(k-3 \mid l) + (k-3\mid l)^2\\
          &  =(k+l-1)^2+2.\\
 \end{align*}

(ii) When $m=0, n=1$,  then again from Proposition \ref{prop444},
  \begin{align*}
 	\dim (H(0,1)\oplus A(k-1 \mid l-1)) \otimes (H(0,1)\oplus A(k-1 \mid l-1))
 	& =(1 \mid 0) + 2(0\mid 1)(k-1 \mid l-1) + (k-1\mid l-1)^2\\
 	&  =(k+l-1)^2.\\
 \end{align*}
(iii) When $m+n\geq 2$, then from Proposition \ref{prop444},  
\begin{align*}
	\dim (H(m,n)\oplus & A(k-2m-1 \mid l-n)) \otimes (H(m,n)\oplus A(k-2m-1 \mid l-n))\\
	& =(4m^2+n^2 \mid 4mn)+2(2m\mid n)(k-2m-1 \mid l-n) +(k-2m-1 \mid l-n)^2 \\
	&  =(k+l-1)^2.\\
\end{align*}
(2) Similarly for $r=0, s=1$, we can find that 
$L \cong H_{m} \oplus A(k-m \mid l-m-1)$ and  $ \dim (L\otimes L) \leq (k+l-1)^2+2. $\\
\smallskip

Let $r+s\geq2$ and assume that the result is true for $r+s-1$. Then for a graded ideal $N$ contained in $Z(L) \cap L^2$ with $\dim N=(1\mid 0)$ (or $\dim N=(0\mid 1)$), we have by  induction hypothesis and Proposition \ref{propp3.2}, 
\[\dim (L\otimes L) \leq  \dim (L/N\otimes L/N) + \dim (L\otimes N ),\]
and 
\[\dim (L/N\otimes L/N)\leq (k+l-1-(r+s-1))(k+l-2)+2.\]
Now from Proposition \ref{pro3333} $$\dim (L\otimes N )=\dim (L/L^2\otimes_{mod} N).$$
Therefore, 
   \[ \dim (L\otimes L) \leq (k+l-(r+s)) + (k+l-(r+s))(k+l-2)+2\]
   \[=(k+l-(r+s))(k+l-1)+2. \]

\end{proof}

\end{document}